\documentclass[a4paper,oneside]{amsart} 

\usepackage[top=4cm, bottom=4cm, left=3cm, right=2.5cm]{geometry}
\usepackage{amsmath, amssymb,amscd,amsthm,bbold}
\usepackage{thmtools}
\usepackage{color}
\usepackage{epsfig}
\usepackage{soul}
\setstcolor{red}
\sethlcolor{yellow}

\usepackage{enumitem}
\usepackage{cprotect}
\usepackage{tikz}
\usepackage{hyperref}

\usepackage[backend=biber]{biblatex}

\input{fitch.sty}

\theoremstyle{definition}

\newtheorem{theorem}{Theorem}[section]

\newtheorem{definition}[theorem]{Definition}

\newtheorem{proposition}[theorem]{Proposition}

\newtheorem{corollary}[theorem]{Corollary}

\newtheorem{example}[theorem]{Example}

\renewcommand{\iff}{\Leftrightarrow}

\newcommand\OO{\mathcal{O}}

\newcommand\RR{\mathbb{R}}

\newcommand\PP{\mathbb{P}}

\newcommand{\Rej}{\text{Reject }}

     \usepackage[top=4cm, bottom=4cm, left=3cm, right=2.5cm]{geometry}

\bibliography{references_severe}

\begin{document}
\title{Confirming the Null: Remarks on Equivalence Testing and the Topology of Confirmation}
\author{Reid Dale}
\address{Stanford University}
\email{reiddale@berkeley.edu}
\date{July 2023}
\maketitle
\begin{abstract}
Null Hypothesis Statistical Testing is a dominant framework for conducting statistical analysis across the sciences. There remains considerable debate as to whether, and under what circumstances, evidence can be said to be confirmatory of a null hypothesis. This paper presents a modal logic of short-run frequentist confirmation developed by leveraging the duality between hypothesis testing and statistical estimation.

It is shown that a hypothesis is confirmable if and only if it satisfies the topological condition of having nonempty interior. Consequently, two-sided hypotheses are not statistically confirmable owing to defects in their topological structure. Equivalence hypotheses are, by contrast, confirmable.

\end{abstract}

\newpage
\tableofcontents

\section{Is there a Frequentist Theory of Short-Run Confirmation?}

The question of whether evidence can be in support of or confirmatory of a scientific theory is of critical importance in the philosophy and interpretation of science. Scientific theories often guide actual decisionmaking, so having a theory of evidentiary support in favor of a theory is critical. A highly developed candidate for confirmation theory is given by Bayesian confirmation theory \cite{earman1992bayes}. Confirmation within the paradigm of frequentist statistics is, by contrast, far from settled. There is a lack of consensus as to whether frequentist statistics admits a formal theory of confirmation and, if it does, what the necessary and sufficient conditions for confirmation of a hypothesis are.

By some accounts, hypotheses are never confirmed, only retained or rejected according to a default logic of hypothesis testing \cite{lehmann_testing_2022}. On other accounts, such as those of Frick \cite{frick_accepting_1995} and Mayo \cite{mayo2018severe}, a hypothesis is only to be accepted or considered severely tested just in case a good-faith effort to refute the hypothesis with a high chance of detecting an error in the hypothesis, if  indeed there was one, was performed. By contrast, Lakens \cite{lakens2017equivalence} has recently claimed that two-sided hypotheses are not confirmable. 

In the present paper I develop a modal logic of evidentiary support or confirmation in the short run, i.e. for finite sets of evidence $E$. In this logic, there are three possible results of a test of a hypothesis $H$: $H$ can be confirmed, rejected, or the test can be simply inconclusive. This theory depends crucially on the duality between hypothesis testing and estimation using confidence regions. On this account, the result of a statistical test is determined by the set-theoretic relationship between a hypothesis $H$ and a $(1-\alpha)$ confidence region $R_{\alpha}$: a hypothesis is confirmed just in case $R_{\alpha} \subseteq H$, rejected if $R_{\alpha} \subseteq H^c$, and is inconclusive if $R_{\alpha}$ nontrivially intersects both $H$ and its complement $H^c$. The criteria for whether a hypothesis is in-principle confirmable can be characterized by a simple topological property of the hypothesis $H$, namely that $H$ contains nonempty interior. In particular, two-sided hypotheses of the form $H: \theta = \theta_0$ in topological spaces such as $\RR$ are not confirmable, while equivalence and non-inferiority hypotheses of the form $H: \theta \in [L, U]$ are confirmable. This suggests that the standard two-sided hypotheses are unduly emphasized within scientific practice: while refutable, they are not confirmable. 

Prior work on the relationship between the topological structure of a hypothesis $H$ and its statistical verifiability was done by Genin and Kelly, who gave a topological characterization of hypotheses 
that are guaranteed to be confirmed if true \cite{genin_topology_2017}. By contrast, this work 
focuses on characterizing hypotheses which, in principle, are statistically confirmable by \textit{some} data.

Finally, the logic of confirmation developed in this paper is used to give a formal semantics for Mayo's notion of severe testing and the Full Severity Principle. Mayo has argued in favor of a ``meta-statistical principle''
that ``ensure that only statistical hypotheses that have passed severe or probative tests are inferred from the data'' by way of criterion of severity that ``avoid[s] classic fallacies from tests that are overly sensitive, as well as those not sensitive enough to particular errors and discrepancies'' \cite{mayo_severe_2006}. The fallacies Mayo speaks of are an artifact of attempting to confirm two-sided hypotheses. On my account, a theory $\tau$ is (partially) confirmed just in case an associated hypothesis $H_{\tau}$ expressing the empirical adequacy of the theory is confirmed.

\section{Doxastic Interpretations of Hypothesis Testing}

Null Hypothesis Statistical Testing (NHST) is a dominant mode of scientific inference. In NHST, given a prespecified significance level $\alpha$ the standard way to interpret an observed p-value $> \alpha$ for a null hypothesis $H_0$ is given the gloss of ``retaining the null $H_0$'' or ``failure to reject $H_0$'' at level $\alpha$. We assume that $\Theta$ is a parametric family of probability distributions on sample space $\Omega$, $H_0 \subset \Theta$ is a set of parameters, and that the alternative hypothesis $H_1 = H_0^c$ is the complement of $H_0$ within $\Theta$. We further assume that evidence $E$ is drawn independently and identically distributed from some distribution $\theta\in \Theta$. For each $n$, this data uniquely determines a probability measure $\PP_{\theta}$ on $\Omega^n$ given by $\theta^{\otimes n}$.

\subsection{The Default Account of Hypothesis Testing} \label{default-orthodox}

Lehmann/Romano \cite{lehmann_testing_2022} is a standard, orthodox frequentist reference for hypothesis testing. For them, hypothesis testing is a tool to solve the following decision problem: on the basis of collected evidence $E$, should we \textit{accept} or \textit{reject} the hypothesis? As they describe it:

\begin{quote} 
We now begin the study of the statistical problem that forms the principal subject of this book, the problem of hypothesis testing. As the term suggests, one wishes to decide whether or not some hypothesis that has been formulated is correct. The choice here lies between only two decisions: accepting or rejecting the hypothesis. A decision procedure for such a problem is called a test of the hypothesis in question. (Page 61)
\end{quote}

Salient to this framing of hypothesis testing are the following features:
\begin{enumerate}
    \item (Asymmetric Roles of Null and Alternative Hypotheses) There is a specified hypothesis $H_0$, the \textit{null} hypothesis, that will be accepted \textit{unless overwhelming evidence overrides it,} and
    \item (Bivalent Decision Problem) There are only two possible values for the decision problem: \textit{accept} or \textit{reject}.
\end{enumerate}

The null hypothesis $H_0$ serves as the \textit{default} output of the test: if the test does not give sufficient evidence against $H_0$, $H_0$ is accepted. For Lehmann and Romano, the design of the test of $H_0$ should be related in some way to the investigator's attitude regarding the truth or plausibility of $H_0$, and should in part be calibrated as a function of the level of evidence required to shake their belief in the hypothesis:

\begin{quote}
    Another consideration that may enter into the specification of a significance level is the attitude toward the hypothesis before the experiment is performed. If one firmly believes the hypothesis to be true, extremely convincing evidence will be required before one is willing to give up this belief, and the significance level will accordingly be set very low. 
    (Page 63)
\end{quote}
On this account of hypothesis testing, the investigator is intended to set the significance of the test as a function of his or her epistemic attitude towards the null hypothesis, resulting in an asymmetry between the null and alternative hypothesis. This epistemic attitude is \textit{not} to be confused with the notion of a prior probability assigned to the hypothesis, but rather with the investigator's tolerance for a Type I error (i.e. the probability that the test mistakenly rejects the null hypothesis).

Many introductory texts in statistics give an interpretation of hypothesis testing in this manner: one \textit{retains} $H_0$ in case $p > \alpha$. Wasserman \cite{wasserman2006all}, for instance, gives this interpretation, likening hypothesis testing to a criminal trial:
\begin{quote}
[H]ypothesis testing is like a legal trial. We assume someone is innocent unless the evidence strongly suggests that he is guilty. Similarly, we retain $H_0$ unless there is strong evidence to reject $H_0$. (Page 150)
\end{quote}

In this telling, the principle of ``presumption of innocence until proven guilty" plays the same formal role as ``retaining the null unless it is rejected with significance $\alpha$." Such principles are \textit{default} principles in the sense that they normatively constrain one's beliefs or actions to a default position (acquittal of the defendant in the case of a criminal trial and retention of the null in the case of hypothesis testing). For example, the presumption of innocence can be interpreted as constraining a factfinder against convicting a defendant in the special case that the prosecution provides \textit{no} evidence of the guilt of the defendant, regardless of the factfinder's underlying personal credence regarding the defendant's guilt. 

There are two major problems with the default account of hypothesis testing that render it implausible. First, and most simply, the investigator need not render a decision in favor of either $H_0$ or $H_1$, by contrast with the decision problem posed in a criminal trial. One may have \textit{intermediate attitudes} toward the hypothesis $H_0$ both prior to and following the experiment. For example, one could for example \textit{affirm} $H_0$, be \textit{ambivalent to} $H_0$, or \textit{reject} $H_0$ in response to evidence $E$. Such a system would require at minimum three attitude towards a hypothesis, more than the two given in the orthodox account. 

Secondly, simply \textit{declaring} $H_0$ to be the null hypothesis does not \textit{warrant} its being taken as default. This issue is most salient when considering the notion of Type II errors, which occurs when a test fails to reject a false hypothesis. On the default view of hypothesis testing, failing to reject $H_0$ results in accepting $H_0$. By designating $H_0$ as the default, a Type II error \textit{is} a genuine error in belief. But it is an error only if one accepts the normative force of the inference to the default in case of a failed test. Implicit in referring to a Type II error as an \textit{error} is an underlying model of default logic as \textit{the} logic of null hypothesis statistical testing. 

To avoid such errors therefore requires a theory of warrant for a default hypothesis. In practical settings, however, warrant is typically not given. It is standard practice to test the two-sided null hypothesis of ``no difference'' even when the investigator has reason to believe that this hypothesis is false. Consider the two-sided null hypothesis is $H_0: \theta = 0$, where $\theta = \mu_A - \mu_B$ is the difference in mean between two populations. Typically, justification for $H_0$ is not given, and $H_0$ may itself be implausible at the outset of the investigation. For example, suppose that I already have knowledge that the quantity $\mu_i$ is causally affected by some property $P$, and that the base rate of $P$ between groups $A$ and $B$ differs substantially. Then the null hypothesis $H_0$ can be plausibly doubted along these lines, so is not a reasonable default position to take, \textit{despite} its apparent symmetry between $A$ and $B$. 

This is not to say that there are no situations in which the default account of hypothesis testing may be of value. The default account does provide a discursive framework by which claims made and believed by an interlocutor can be treated as default and subjected to scrutiny through statistical analysis. However, this discursive paradigm does not exhaust the modes of scientific discourse, which may for example include scenarios where the investigator has no prior attitude toward either $H_0$ or $H_1$ and would like to assess whether \textit{either} of them has strong evidence in its favor.

Through its reliance on the specification of a default hypothesis and the bivalence of the testing procedure, the default account of hypothesis testing fails to provide a framework that adequately models the various epistemic attitudes an investigator might reasonably have towards a hypothesis.  

\subsection{The Falsificationist Account of Hypothesis Testing and Frick's Counterargument}

Owing to problems in the interpretation of $p$-values, and motivated by Popper's falsificational account of scientific theories, some maintain that the only normative role that hypothesis testing plays is that of refuting the null hypothesis. This account stands in contrast to the default account in terms of the admissible values for a test: either $H_0$ is rejected, or $H_0$ is not rejected. Thus, a Type II error is no longer an error in the investigator's updated belief in $H_0$. But, like the default account of hypothesis testing, there is an asymmetric role between the null and alternative hypotheses and the test procedure is bivalent.



In a highly cited paper, Frick reconstructs an argument as to \textit{why} people might think the null hypothesis ought never be accepted. He writes:

\begin{quote}
    In any given interval (e.g., from -10 to +10), there are an infinite number of real numbers. When the answer to a question could be any real number in an interval, and when the probability of each real number being the answer is approximately the same, the finite amount of probability must be spread over the infinite number of
 real numbers. Each individual real number must receive essentially zero (i.e., infinitesimal) probability, and only an interval of real numbers can receive a nonzero probability. (A nonzero probability is one greater than essentially zero; winning the lottery, for example, has a nonzero probability.) 
 
 This logic would be applied (incorrectly) to the problem of accepting the null hypothesis in the following way: The question concerns the size of the effect one variable has on another, with size zero corresponding to the null hypothesis. Zero is just one point in the interval of infinite possible answers, so the probability of the effect being exactly 0.00000\dots (with the zeros being carried out to an infinite number of decimal places) is essentially zero. Therefore, the null hypothesis is impossible.  (Pages 132-133)
\end{quote}

It is worth highlighting that only two-sided hypotheses are contained within the scope of this argument, since hypotheses of the form $H_0: \theta \in [L,U]$ with $L< U$ can have positive measure. From a frequentist perspective, the hypothetical argument that Frick describes is indeed untenable. Frequentists do not recognize the meaningfulness of a nontrivial \textit{probability} measure on the space of hypotheses.

To see why, let us first formalize the argument given above in a Bayesian paradigm. Let $\Theta = [0,1]$ and let $H_0: \theta = \frac{1}{2}$. Then with respect to the uniform prior, or indeed any atomless prior, on $\Theta$, $\PP(H_0) = 0$. Thus $\PP(H_0|E) = 0$ for any evidence accrued, and so $H_0$ is unconfirmable even if true. Unfortunately, this argument is not admissible within a frequentist paradigm.

For a frequentist, probability statements of the form $\PP(H)$ can be formally defined, but are not epistemically useful. Let $\theta_{\top}$ be the true value of the parameter $\theta$. Given $\theta_{\top}$, one can assign the probability that $H$ is true to be  \[\PP_{\top}(H) = \left\{ \begin{array}{ll}
         1 & \mbox{if } \theta_{\top}\in H ;\\
        0 & \mbox{if } \theta_{\top}\notin H.
        \end{array} \right.  \]
$\PP_{\top}$ assigns a probability of $H$ equal to its truth value: $H$ has positive probability if and only if $H$ has probability $1$ if and only if $\theta_{\top}\in H$. With respect to this measure, $\PP(\{\theta_{\top}\}) = 1$ and so $\PP(\{\theta_{\top}\}|E) = 1$. Thus, there \textit{exists} a probability measure for which $\theta_{\top}$ is correctly confirmed conditional on any evidence. At first glance, this may suggest that $\{\theta_{\top}\}$ is confirmable. But this is incorrect. One cannot know the probability assignments of $\PP_{\top}$ \textit{a priori}: exhaustive knowledge of the probability assignments of $\PP_{\top}$ would entail knowledge of the true value $\theta_{\top}$. 

It would seem therefore that arguing that two-sided hypotheses $H_0: \{\theta\}$ are unconfirmable does not fit cleanly within a measure-theoretic framework for the frequentist.

Frick's argument in favor of the confirmability of the null relies on a probability theoretic grounding inadmissible to the frequentist:

\begin{quote}
However, the argument above does not necessarily
apply to accepting the null hypothesis. A few special numbers can be allocated a nonzero probability before spreading the remaining probability over the remaining real numbers. In the scales used by psychologists, zero is usually a special, meaningful number. Therefore, the number zero can be assigned a nonzero probability of occurring without violating any rules of probability. 
(Page 133)
\end{quote}

Frick correctly rebuts the Bayesian argument given above: for a Bayesian investigator it is absolutely correct that ``the number zero can be assigned a nonzero probability of occurring without violating any rules of probability," and that one can ``allocat[e] a nonzero probability'' to the value of the effect size being exactly zero. These considerations suggest that one may, in certain circumstances where there is expert knowledge, be justified in having a prior probability of truth assigned to a two-sided hypothesis greater than $0$, and therefore be confirmable in the Bayesian sense as the conditional probability $\PP(H_0|E)$ has the possibility of becoming arbitrarily close to $1$.

Despite the apparent success of Frick's argument, it ultimately rests on the assignment of prior probabilities to hypotheses inadmissible to the frequentist. In frequentist inference, there is no sense in talking about a prior \textit{probability} of the hypothesis in question.

In the following section I provide an alternative, frequentist analysis of the possibility of confirming a two-sided hypothesis. I will argue that the conclusion Frick responds to is correct: two-sided hypotheses in continuous parameter spaces cannot be confirmed. This is not due to any measure theoretic consideration, but instead is due to the \textit{topological} defects of two-sided hypotheses. The argument I will give formalizes Lakens' claim that ``it is statistically impossible to support the hypothesis that a true effect size is exactly zero.'' \cite{lakens2017equivalence}.

\subsection{A Middle Path: Confidence Intervals, Decision, and Indecision}

\subsubsection{Duality Between Confidence Intervals and Hypothesis Testing}

Recall from \cite{keener2010theoretical} (Definition 9.17) that a two-sided $1-\alpha$ confidence interval (or, more generally, a \textit{confidence region}) is a random set $c(E)$ which is a function of the collected data (or ``evidence'') $E$ such that for each $\theta\in\Theta$
\begin{equation}
    \PP_{\theta}(\theta \in c(E)) \geq 1-\alpha.
\end{equation}
I emphasize that this is \textit{not} a probability statement about $\theta$; rather, it is a probability statement about the properties of the stochastically-generated object $c(E)$ \cite{wasserman2006all}.\footnote{The dependence of the event $\theta\in c(E)$ on the stochastic nature of $c(\cdot)$ can be readily seen by writing the event as $X_{\theta} = \{ E\in \Omega^n\,|\, \theta\in c(E)\}$, which is a subset of $
\Omega^n$.} 

For many common tests, rejection of $H_0$ at level $\alpha$ is equivalent to a property of a $1-\alpha$ confidence interval $c(E)$:
\begin{equation} 
\Rej H_0 \Leftrightarrow c(E) \subseteq H_0^c
\end{equation}
In other words, one rejects $H_0$ at the $\alpha$ level in case the associated confidence interval $c(E)$, a nonempty open subset of $\Theta$, is fully contained within the complement of $H_0$. This apparent reduction of hypothesis testing to set-theoretic relationships between confidence intervals and subsets of the underlying parameter space render it amenable to modal interpretation.\footnote{An explanation of the duality between estimation of the confidence interval and hypothesis testing can be found in \cite{keener2010theoretical} section 12.4: ``Duality between testing and interval estimation.''}

By contrast with the previously discussed interpretations of hypothesis testing, the utilization of confidence intervals allows us to remedy the epistemic issues discussed above. 

\begin{enumerate}
    \item (Symmetric Role of the Null and Alternative Hypotheses) The null and alternative hypotheses $H_0$ and $H_1$, neither of which will be accepted at default.
    \item (Multivalent Decision Problem) There are three potential values for the decision problem for each hypothesis $H_i$: \begin{itemize}
    \item Confirm $H_0$: $\square H_0$,
    \item Confirm $H_1$: $\square H_1$, and
    \item Neither Confirm $H_0$ or $H_1$: $\lozenge H_0 \wedge \lozenge H_1$.
    \end{itemize}
\end{enumerate}

The truth condition for each of the possible outcomes of a test can be given directly in terms of the set-theoretic relationship between the (prespecified) confidence interval process $c$, evidence $E$, and hypotheses $H_0$ and $H_1$. Intuitively, we say that a hypothesis is confirmed provided that a sufficiently leveled confidence interval is \textit{fully contained within} $H_i$. The three possible set-theoretic relationships can be visualized as follows:

\begin{figure}[htbp]
  \centering
   \includegraphics[width=6in]{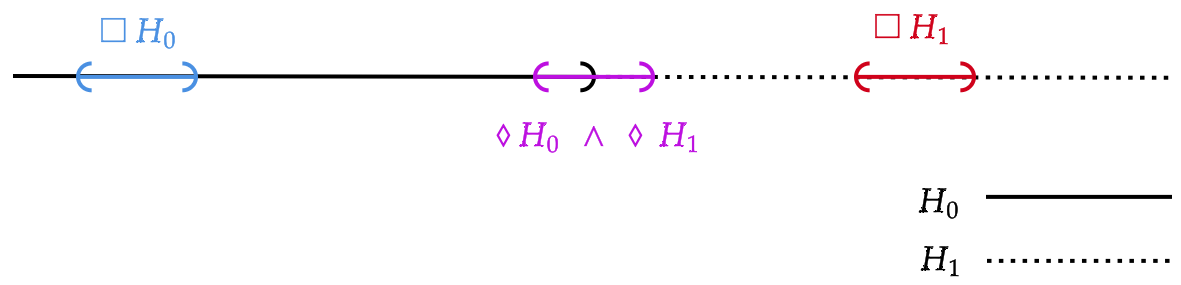}
  \caption{Satisfaction Conditions for Confirmation}
\end{figure}

More formally, we have 
\begin{enumerate}
    \item $\square H_i$ if and only if $c(E) \subseteq H_i$, and 
    \item $\lozenge H_0 \wedge \lozenge H_1$ if and only if $c(E) \cap H_0 \neq \varnothing$ and $c(E) \cap H_1 \neq \varnothing$.
\end{enumerate}

Confirmation of a hypothesis $H_i$ in this sense is therefore witnessed by a sufficiently narrow, well-located confidence region. Indecisiveness of the test occurs when the confidence interval is not fully contained within \textit{either} $H_0$ or $H_1$. 

Reducing the problem of hypothesis testing to containment properties of confidence regions results in a particularly natural and simple semantics for a multivalent logic of hypothesis testing. In the remainder of this section we illustrate how to think of the possibility of confirmation through the lens of hypothesis testing through three examples. 

\subsubsection{Two-Sided Hypotheses} \label{twosided_defect}
To begin with an example, consider $\Theta = \RR$ and suppose
\begin{equation} H_0: \theta = \theta_0 \text{ vs. } H_1:\theta\neq \theta_0\end{equation}
We reject $H_0$ provided that the $1-\alpha$ confidence interval excludes $\{\theta_0\}$:
\begin{equation} \Rej H_0 \Leftrightarrow c(E) \cap \{\theta_0\} = \varnothing.\end{equation}
Likewise, by swapping the role of the null and alternative hypotheses,
\begin{equation} \Rej H_1 \Leftrightarrow c(E) \cap \{\theta_0\}^c = \varnothing.\end{equation}

The hypothesis $H_0$ can be rejected: if, for example, $c(E) = (\theta_0 + \epsilon, \theta_0 + \eta)$ with $\epsilon, \eta > 0$ then $c(E) \cap \{\theta_0\} = \varnothing$ and we can reject $H_0$. 

By contrast, $H_1$ \textit{cannot} be rejected. There are many proofs of this, but I focus here on a topological one: the confidence interval $c(E)$ is a nonempty open subset of $\RR$, and $\{\theta_0 \}$ is a single point. Thus $\{\theta_0 \}^c$ is a dense subset of $\RR$ and therefore intersects $c(E)$, so $H_1$ cannot be rejected. This entails that $H_0$ cannot in principle be \textit{confirmed} in this testing paradigm.

The above line of reasoning can be made more simple by converting it to the notation of modal logic. We begin by observing that ``failure to reject $H_0$'' expresses a kind of \textit{possibility} or \textit{plausibility}\footnote{The interpretation of confidence intervals modally, as a range of plausible values for the parameter $\theta$, has been advocated by Smithson \cite{smithson2003confidence}}: the evidence $E$ generating the confidence interval $c(E)$, when subjected to the test $T$ (i.e. $c(E) \subseteq H_1$), left open the possibility of $H_0$ (at least at the level $\alpha$). Thus, we get the following set of equivalences.

\[ \begin{array}{lcl}
\lozenge H_0 & \Leftrightarrow &\neg (\Rej H_0) \\
            & \Leftrightarrow &\neg (c(E) \cap H_0 = \varnothing) \\
             & \Leftrightarrow &(c(E) \cap H_0 \neq \varnothing) \\
             & \Leftrightarrow & \neg (c(E) \subseteq H_1)
\end{array} \] 

Adopting the typical modal definition $\square = \neg \lozenge \neg$ we get a definition for the \textit{confirmation} of the null hypothesis given evidence:

\[ \begin{array}{lcl}
\square H_0 & \Leftrightarrow & \neg\lozenge\neg H_0 \\
 & \Leftrightarrow & \neg \lozenge H_1 \\
 & \Leftrightarrow & \neg(\neg (c(E) \subseteq H_0)) \\
 & \Leftrightarrow &  c(E) \subseteq H_0. 
\end{array} \] 

In plainer language, one confirms the null $H_0$ just in case one rejects the alternative $H_1$. This means that confirmation is a distinguished form of of null result: one in which not only is the $p$-value of the test greater than $\alpha$ since $c(E) \cap H_0 \neq \varnothing$, but one in which the alternative hypothesis is simultaneously \textit{rejected} at level $\alpha$. Thus, confirmation is a more stringent condition than merely finding a null result. 

However, the very possibility of the confirmation of a hypothesis hinges on the topological structure of the hypothesis. Returning to the original two-sided example, if $H_0 = \{\theta\}$ and $H_1 = \RR \setminus H_0$ then 
\[ \Vdash \neg \square H_0;\]
i.e. it is a validity that $H_0$ is not confirmable relative to $H_1$. Equivalently,
\[ \Vdash \lozenge H_1. \]
Two-sided hypotheses can therefore be considered non-confirmable in a precise sense. 

\subsubsection{Two One-Sided Hypotheses}

Let $M>0$. Consider the Two One-Sided Test (TOST) hypothesis
\begin{equation} H_0: |\theta - \theta_0| > M \text{ vs. } H_1: |\theta - \theta_0| \leq M\end{equation}
We reject $H_0$ provided that the $1-\alpha$ confidence interval excludes $(-\infty, -M) \cup (M, \infty)$:
\begin{equation} \neg\lozenge H_0 \Leftrightarrow c(E) \cap H_0 = \varnothing.\end{equation}
By contrast with two-sided hypotheses, the TOST paradigm admits a nontrivial notion of confirmation of the null \textit{and} the alternative:
\begin{equation} \Vdash \square H_0 \Leftrightarrow c(E) \cap [-M,M] = \varnothing.\end{equation}
and
\begin{equation} \Vdash \square H_1 \Leftrightarrow c(E) \cap \Big( (-\infty, -M) \cup (M, \infty) \Big) = \varnothing.\end{equation}

both of which are possible.




In the following section we build up a general framework for reasoning about this kind of confirmation.

\section{Formal Semantics for the Modal Logic of Hypothesis Testing}

\subsection{Syntax}

The salient features of this logic require only a very sparse syntax. We may suppose that we have a single sentence $\{H_0\}$, the negation connective $\{\neg \}$, and the possibility operator $\lozenge$. We define $H_1 = \neg H_0$ and $\square H_i = \neg \lozenge \neg H_i$. We adopt the usual definitions of satisfaction for the remaining Boolean connectives, although they are not our primary focus. 

\subsection{Semantics} \label{semantics-naive}

Let $\Theta$ be a space of parameters, and $H_0 \subset \Theta$. Let $\Omega$ be the sample space of observations. By evidence $E$, I mean a finite string of observations $E\in \Omega^{<\omega}$. The semantics of modal logics are commonly exhibited in terms of Kripke frames by defining an accessibility relation. Kripke semantics for the logic described in this section are provided in Appendix \ref{kripke-appendix}. To keep the core notions of statistical confirmation accessible to a wide audience, we present a more direct equivalent semantics.

A \textit{possible world} is a pair $(\theta, E)$ where $\theta\in \Theta$ and $E\in \Omega^{<\omega}$.
Let $H_0 \subseteq \Theta$. We write 
\begin{equation} w = (\theta_w, E_w) \Vdash H_0 \iff \theta_w \in H_0.\end{equation}

Satisfaction of the negation of a sentence $S$ is defined via:

\begin{equation} w \Vdash \neg S \iff \neg (w \Vdash S).\end{equation}

To define the semantics on modal operators, we need a notion analogous to that of a confidence interval. By a confidence function I mean a function $c: \Omega^{<\omega} \to 2^{\Theta}$, which assigns to evidence $E$ a subset $c(E) \subset \Theta$. At the outset we do not require anything further of $c$, not even that $c$ be nonempty: we can interpret $c(E) = \varnothing$ as saying that the evidence $E$ is incompatible with all $\theta\in \Theta$. However, we say that $c$ is non-refuting provided $c(E) \neq \varnothing$ for all $E\in \Omega^{<\omega}$. 

Given a confidence function $c$ on $\Omega^{<\omega}$, we can define the semantics of possibility operators:

\begin{equation} w \Vdash \lozenge H_0 \iff c(E_w) \cap H_0 \neq \varnothing.
\end{equation}

Routine calculation shows that 

\begin{equation} w \Vdash \square H_0 \iff c(E_w) \subseteq H_0.
\end{equation}

Note that $\square$ is generally irreflexive (i.e. the reflexive axiom $\square \phi \rightarrow \phi$), underlying a key feature of statistical confirmation. In general,

 \begin{equation} w \Vdash \square H  \not\Rightarrow w \Vdash H. \end{equation}

For example, let $H_0 = \{\theta_0\}^c$ and $w = \{\theta_0, E\}$ such that $\theta_0\notin c(E)$. Then $w \not\Vdash H$ since $\theta_w = \theta_0$, but $w \Vdash \square H_0$ since $c(E) \subseteq H_0$ as $\theta_0\notin c(E)$.

A final observation is that there are precisely three satisfiable outcomes for the satisfaction of the modal sentences:
\begin{proposition} \label{three-possibilities}
Let $c$ be non-refuting. Then for every $H_0$ and world $w$, exactly one of the three sentences is true: $w \Vdash \square H_0$, $w\Vdash \square H_1$, or $w\Vdash \lozenge H_0 \wedge \lozenge H_1$.
\end{proposition}
\begin{proof}
It suffices to show that $\square H_0 \wedge \square H_1$ is false, since $\lozenge H_0 \wedge \lozenge H_1 \iff \neg (\square H_0 \vee \square H_1)$ and $P\vee Q \vee \neg(P \vee Q)$ is a propositional tautology.

Suppose for contradiction that $w\Vdash \square H_0 \wedge \square H_1$. Then $c(E_w) \subset H_0$ and $c(E_w)\subset H_0^c$, so $c(E_w) = \varnothing$. This contradicts $c$ being non-refuting.
\end{proof}

\subsection{The Topological Case}

In this section we add additional constraints to $\Theta$ and $c$ in line with the topological properties of the confidence interval in parametric statistics. The reader unfamiliar with topological terminology is advised to consult the classic textbook \cite{Munkres2000Topology} for the definitions of relevant terminology.

\begin{definition}
Let $\Theta$ be a topological space and $\OO(\Theta)$ the collection of open sets. A open-valued confidence function is a map $c: \Omega^{<\omega} \to \OO(\Theta)$. 
\end{definition}

Open-valued confidence functions are natural objects. A confidence function being open-valued is just to say that if $\theta$ is plausible with respect to $E$ (i.e. $\theta \in c(E)$), then sufficiently small perturbations of $\theta$ are \textit{also} plausible with respect to $E$; in topological terms, this means precisely that there is an open set $U\in \OO(\Theta)$ such that if $\theta' \in U$ then $\theta'$ is ``sufficiently close'' to $\theta$.\footnote{It is worth reviewing a few subtleties pertaining to the use of open sets. Regarding discrete sets of parameters $\Theta$, we may assign $\Theta$ the discrete topology. In this case, every subset of $\Theta$ is both closed and open.  Regarding the use of open-valued confidence sets, one might argue that in the construction of confidence intervals via the bootstrap, the confidence interval should be a \textit{closed} interval with nonempty interior so that the confidence interval nominally contains the $1-\alpha$ endpoints. To assuage this concern, I note that if one replaces ``open-valued'' with the constraint that $c(E)$ is the closure of a nonempty open set, all subsequent results in this section remain true as the proofs require only that $c(E)$ have nonempty interior.} 

Under this condition, we can give a sufficient condition for when a null hypothesis $H_0$ is nonconfirmable:

\begin{proposition} \label{emptyinterior-nonconf}
Let $\Theta$ be a topological space and $c$ a open-valued non-refuting confidence function. Suppose that $H_0$ has empty interior. Then for every $w$, $w \Vdash \lozenge H_1$; i.e.  $w\Vdash \neg \square H_0$.

In particular, for a two-sided hypothesis $H_0: \theta = \theta_0$, no world $w$ satisfies $w\Vdash \square H_0$.
\end{proposition}
\begin{proof}
Since $H_0$ has empty interior, $H_1$ is dense. Since $c$ is open-valued non-refuting, $c(E_w)$ is an open nonempty set for all $E_w$. Thus $c(E_w) \cap H_1 \neq \varnothing$. Thus $w\Vdash \lozenge H_1$; equivalently, $w\Vdash \neg \square H_0$.
\end{proof}

This observation extends to a characterization of the realizability of $\square H_i$ and $\lozenge H_i$ by possible worlds.

\begin{theorem}\label{satisfiability-criteria}
Let $\Theta$ be a connected topological space, $c$ an open-valued confidence function, $\Omega$ the sample space. 

Suppose further that the image $c(\Omega^{<\omega}) \subset 2^{\Theta}$ satisfies 
\begin{enumerate}
    \item ($\Theta$-Exhaustiveness) for each $\theta\in \Theta$ there exists an $E_{\theta}\in \Omega$ such that $\theta \in c(E_{\theta})$.
    \item (Precision) For each open $U\in \OO(\Theta)$ there is evidence $E_U \in \Omega^{<\omega}$ such that $c(E_U)$ is nonempty and $c(E_U)\subseteq U$.
\end{enumerate}

Suppose that $H\subseteq \Theta$ is nonempty. Then
\begin{enumerate}
    \item (Satisfiability of $\square H$) There is a world $w$ such that $w\Vdash \square H$ if and only if $H$ has nonempty interior, and
    \item (Satisfiability of $\lozenge H \wedge \lozenge H^c$) There is a world such that $w\Vdash \lozenge H \wedge \lozenge H^c$ if and only if $H^c$ is nonempty.
\end{enumerate}
\end{theorem}

\begin{proof}
First, let $H$ have nonempty interior $H^{\circ}$. By precision of $c$, there is some $E_H$ such that $c(E_H) \subseteq H^{\circ}\subseteq H$. Select any $\theta\in\Theta$ and let $w = (\theta,E_H)$. Then $w\Vdash \square H$. Conversely, suppose that $H$ has empty interior. Then by \ref{emptyinterior-nonconf}, for every $w$ we have $w\Vdash\lozenge \neg H$, so $w\Vdash \neg \square H$. Thus $\square H$ is not satisfiable. 

Hence, $\square H$ is satisfiable if and only if $H$ has nonempty interior.

Finally, since $H$ and $H^c$ are nonempty and $\Theta$ is connected, the boundary $\partial H = \partial H^c$ is nonempty. Let $\theta_H \in \partial H$. By assumption on $c$, there exists $E_{\theta_H}$ such that $\theta_H \in c(E_{\theta_H})$. Since $\theta_H \in \partial H$, any open set containing $\theta_H$ intersects $H$ and $H^c$ nontrivially. Hence $H\cap c(E_{\theta_H})$ and $H^c \cap c(E_{\theta_H})$ are nonempty. Thus the world $w = (\theta_H, c(E_{\theta_H}))$ has $w\Vdash \lozenge H \wedge \lozenge H^c$.
\end{proof}

The conditions on confidence intervals in Theorem \ref{satisfiability-criteria} are satisfied by two-sided Wald confidence intervals, but not by \textit{one-sided} confidence intervals. One-sided confidence intervals are insufficient because their length is always infinite, but 
precision requires the confidence interval to be containable within arbitrary open sets of finite length given an appropriate stream of finite data.

This topological criterion immediately highlights the difference between two-sided hypothesis testing versus the two one-sided hypothesis testing paradigms. 

\begin{corollary}\label{simplified-satisfiability}
Let $\Theta$, $c$, $\Omega$ be as in \ref{satisfiability-criteria} be a topological space with no open points.\footnote{Containing no open points is entailed in the case that $\Theta$ is connected and not a single point by stronger separation axioms, such as $\Theta$ being Fr\'{e}chet (T1) or Hausdorff (T2).} Then 

\begin{enumerate}
    \item If $H_0 = \{\theta\}$, then $\square H_1$ and $\lozenge H_0 \wedge \lozenge H_1$ are satisfiable, but not $\square H_0$.
    \item If $H_0, H_1$ is such that the interiors $H_0^{\circ}$ and $H_1^{\circ}$ are both nonempty, then $\square H_0$, $\square H_1$, and $\lozenge H_0 \wedge \lozenge H_1$ are all satisfiable. 
\end{enumerate}
\end{corollary}
\begin{proof}
By \ref{three-possibilities}, for each world exactly one of $\square H_0$, $\square H_1$, or $\lozenge H_0 \wedge \lozenge H_1$ is true. The rest follows immediately from \ref{satisfiability-criteria}.
\end{proof}

To assert that a true effect size is \textit{exactly} zero is to posit a hypothesis with nonempty interior, namely, $H_0 = \{\theta\}$ inside a space $\Theta$ where no point is open. By contrast, if $H_0 = [-M,M]$ with $M \in (0,\infty)$ then $H_0$ and $H_1$ are both confirmable.

\section{Revisiting Type I and Type II Errors}

Recall the distinction between a Type I and Type II error from hypothesis testing:

\begin{enumerate}
\item A type I error occurs when we reject $H_0$ despite $H_0$ being true. In our formal symbolism: $\square H_1 \wedge H_0$.
\item A type II error occurs when we fail to reject $H_0$ when $H_1$ is true; i.e. $\lozenge H_0 \wedge H_1$.
\end{enumerate}

Considering the situation in a type II error an \textit{error} seems to hinge upon the understanding of hypothesis testing as a \textit{default logic}: type II errors occur when one retains $H_0$. However, if one adopts a modal language, inferring the \textit{possibility} of $H_0$ in light of the test and accrued evidence $E$ is less an error in the conclusion than the failure of the test to result in a decisive outcome.

To this end, the modally-inspired analogues of Type I and Type II errors for the logic of Two One-Sided Testing are given as follows:

\begin{enumerate}
\item A \textit{decisive error} occurs when $\Vdash \square H_i \wedge H_{1-i}$, and
\item \textit{Indecisiveness} occurs when $\Vdash \lozenge H_i \wedge \neg \square H_i$; or, equivalently, $\Vdash \lozenge H_0 \wedge \lozenge H_1$.
\end{enumerate}

We assume throughout this section that the events $\square H_i$ and $\lozenge H_i$ are all measurable events for each $\theta\in \Theta$, and that two-sided $1-\alpha$ confidence intervals exist for $\Theta$. Note that we may think of the confidence interval $c(E)$ as a random set as a function of the evidence $E$.
 
\subsection{Decisive Error}

The rate of decisive errors is closely related to the coverage of the confidence interval.

\begin{proposition}
Let $c$ be a $(1-\alpha)$ confidence region. Then the $d$-value ("d" for "Decisive") of the test,
\[ d_{\square} = \max\Big(\sup_{\theta\in \Theta_0}\PP_{\theta}(\square H_1), \sup_{\theta\in \Theta_1}\PP_{\theta}(\square H_0)\Big)  \]
satsifies
\[d_{\square} \leq \alpha.\]
\end{proposition}
\begin{proof}

For a given $\theta \in H_{0}$ and level $1-\alpha$
\begin{equation}
\begin{array}{lclr}
\PP_{\theta}\big(\square H_1\big) & = &\PP_{\theta}\big(c(E) \subseteq H_1\big) & \\
 & \leq & \PP_{\theta}\big(c(E) \cap \{\theta \} = \varnothing\big) & \text{(Since } \theta \notin H_1 ) \\
 & \leq & \alpha & \text{(By definition of confidence interval)}
\end{array}
\end{equation}

Replacing $H_0$ with $H_1$ we can conclude that for $\theta \in H_1$, $\PP_{\theta}\big(\square H_0\big) \leq \alpha$. Thus 

\begin{equation}
    \max\Big(\sup_{\theta\in \Theta_0}\PP_{\theta}(\square H_1), \sup_{\theta\in \Theta_1}\PP_{\theta}(\square H_0)\Big) \leq \max(\alpha,\alpha) = \alpha.
\end{equation}

as desired.
\end{proof}

Note that the bound of $\alpha$ may not be sharp: in TOST, for example, a two-sided $(1-2\alpha)$ confidence interval achieves significance $\alpha$ for testing $H_0: |\theta - \theta_0| \geq M$ \cite{lakens2017equivalence}.

\subsection{Power Analysis and Indecisiveness}

Recall that the power function measures the probability of rejecting the null $H_0$, assuming the true parameter is $\theta \in H_1$\footnote{This is not strictly necessary; the power function makes sense for $\theta\in H_0$ as well. But the ``power" $\beta(\theta,n)$ when $\theta \in H_0$ is the probability of rejecting $H_0$, which in that case is the same as the significance of the test at $\theta$.} and $E\in \Omega^n$ is drawn from the distribution $\theta^{\otimes n}$. In other words, 
\begin{equation}
\beta(\theta,n) = \PP_\theta(c(E) \cap H_0 = \varnothing ) =  \PP_\theta(\square H_1).\footnote{Note that the dependence on $n$ is suppressed in this notation: $c(E)$ is a function of $E$, which in turn is an event of $n$ observations drawn from the true data generating process.}
\end{equation}

We can extend this to define the \textit{decisive power function}, which make symmetric the roles of $H_0$ and $H_1$ and measure how frequently the test will yield a decisive outcome of $\square H_0$ or $\square H_1$:

\begin{definition}
The \textit{partial decisive power functions} $\delta_0, \delta_1$, are given by
\begin{equation}
\begin{array}{lcl}
 \delta_0(\theta, n) & = & \PP_\theta(\square H_1), \\
 \delta_1(\theta, n) & = & \PP_\theta(\square H_0). 
\end{array}
\end{equation}

The \textit{decisive} power function $\delta(\theta,n)$ is the probability of reaching a decisive result, given $\theta$:
\begin{equation}
\begin{array}{lclr}
 \delta(\theta, n) & = & \PP_\theta(\square H_0 \vee \square H_1) & \\
 & = & \PP_\theta(\square H_0) + \PP_\theta(\square H_1) & \text{(By disjointness of the events } \square H_0 \text{ and } \square H_1 \text{)}\\
 & = & \delta_1(\theta,n) + \delta_0(\theta,n). & 
\end{array}
\end{equation}
\end{definition}

Observe that for a given null hypothesis $H_0$, the decisive power function $\delta(\theta,n)$ is always at least as large as the usual power function $\beta_0(\theta,n)$:

\begin{proposition} Let $H_0$ be a hypothesis, $\delta$ its decisive power function, and $\beta$ its power function. Then the decisive power function $\delta(\theta,n) \geq \beta(\theta,n)$.
\end{proposition}
\begin{proof} Since $\delta_0 = \beta$ is the power function for $H_0$ and both partial decisive power functions $\delta_i$ take values in $[0,1]$, $\delta(\theta,n) \geq \beta(\theta,n)$.
\end{proof}

This signals that that the power function may fail to capture all of the confirmation-theoretic information encapsulated by a given test. We show a case where this can happen, and one where it doesn't.

\begin{example} (Decisive Power Equals Power for Two-Sided Tests) Let $H_0 = \{\theta \}$ and $\delta(\theta,n)$ its decisive power function. Let $\theta \in H_1$. As we have already established, $\Vdash \neg H_0$, so $\PP_{\theta}(\square H_0) = 0$. Thus
\begin{equation}
\begin{array}{lcl}
 \delta(\theta, n) & = & \PP_\theta(\square H_0 \vee \square H_1) \\
 & = & \PP_\theta(\square H_0) + \PP_\theta(\square H_1) \\
 & = & 0 + \delta_0(\theta,n) \\
 & = & \beta(\theta,n). 
\end{array}
\end{equation}
\end{example}

\begin{example} (Decisive Power Exceeds Power for TOST) Let $H_0 = [-M,M]$ and $\delta(\theta,n)$ its decisive power function. Let $\theta =M$. Then for a sufficiently large $n$ the confidence interval $c(E)$ will be contained inside $H_i$ with positive positive probability for $i = 0,1$. Thus $\delta_0(M,n), \delta_1(M,n) > 0$ and so $\delta(M,n) > \delta_0(M,n) = \beta(M,n)$.\footnote{Observe that in this case the decisive power of the test was witnessed by the possibility of a Type I error at the boundary of the TOST interval. However, the rate of false positives is still bounded above by $\alpha$ by the construction of $c(E)$.}
\end{example}

We now revisit the concept of Type II errors discussed in section \ref{default-orthodox}. By contrast with the default account of hypothesis, the confirmatory logic of hypothesis testing can avoid the pitfalls of Type II errors by not imposing normative constraints on an investigator's beliefs in the event of an inconclusive test. On this view, failing to reject a hypothesis, i.e. $\lozenge H_0$, does not impose any doxastic constraints on an agent: one can continue to believe either $H_0$ or $H_1$. Unlike the default position, a null hypothesis $H_0$ need not be justified \textit{a priori}, as the modal approach does not constrain an agent to hold particular beliefs regarding the null. A Type II error, then, is not an error in an agent's doxis, since the result of the test places no constraints on doxastic states. A test resulting in a determination of $\lozenge H_0 \wedge \lozenge H_1$ is, in this sense, merely \textit{indecisive}. Only when $\Vdash \square H_i$ would any constraints on doxis be imposed. Naturally, in the design of the experiment we would ideally like to assess the probability that the experiment will be indecisive, that is that $\lozenge H_0 \wedge \lozenge H_1$. This is closely related to the decisive power function.

\begin{definition}
The indecisiveness function $i_{\lozenge}(\theta,n)$ is the probability of yielding $\lozenge H_0 \wedge \lozenge H_1$ given $\theta$ and sample size $n$:
\begin{equation}
\begin{array}{lcl}
 i_{\lozenge}(\theta, n) & = & \PP_\theta(\lozenge H_0 \wedge \lozenge H_1) \\
 & = & 1 - \PP_\theta(\square H_0 \vee \square H_1) \\
 & = & 1 - \delta(\theta,n).
\end{array}
\end{equation}
\end{definition}

In the two-sided case, the indecisiveness of the test is precisely the rate of Type II error since the power function and the decisive power function coincide.

A more theoretical example illustrates how the topological properties of $H_0$ relate to indecisiveness.

\begin{proposition}
Let $H_0\subset \Theta$ be dense-codense. Then $i_{\lozenge}(\theta,n) = 1$.
\end{proposition}
\begin{proof}
Since $H_0$ is dense/codense, if $U\subset \Theta$ is open nonempty then $H_0\cap U$ and $H_1 \cap U$ are nonempty. Since confidence intervals are open nonempty, $\Vdash \lozenge H_0 \wedge \lozenge H_1$, so $i_{\lozenge}(\theta,n) = 1$.

\end{proof}

\section{Equivalence Testing is a Model of Mayo's Severe Testing}

\subsection{Two One-Sided Test on Loss Functions are a Confirmation Logic of Empirical Adequacy} 

By a predictive theory $\tau$, I mean a theory which makes predictions---either deterministic or probabilistic---of the outcomes of events under a certain experimental context. That is, given some experimental setup $X$, there is a function $f_{\tau}$ which generates predictions $f_{\tau}(X)$. Let $A(X)$ be the actual outcome of the experimental setup. We assume that the empirical adequacy of the theory $\tau$ is assessed according to a loss function $L$ measuring the discrepancy between the predictions of $\tau$ and the actual outcomes: $L(f_{\tau}(X), A(X)) \in \RR^{\geq 0}$. Our strategy for confirming $\tau$ relative to the experiment $X$ is to estimate and bound a statistic $\theta = \theta(L,\tau)$ of the loss function.

Given a tolerance for error---an \textit{adequacy margin}---$M>0$, we can perform the following test on $\theta$ at level $\alpha$:
\begin{equation}
H_0: |\theta| \leq M \text{ vs. } H_1: |\theta| > M.
\end{equation}
Formally, this test is the test of TOST with the null and alternative hypotheses inverted. We interpret the null $H_0$ as expressing the empirical adequacy of the theory $\tau$ up to level $\alpha$. As discussed in \ref{simplified-satisfiability}, a hypothesis $H_0$ with nonempty interior admits the possibility of confirmation up to level $\alpha$. Thus, relative to an experimental procedure $X$, loss function $L$, and adequacy margin $M$ there is a well-defined sense in which we can say that the evidence $E$ \textit{statistically confirms} the empirical adequacy of the theory $\tau$. I want to emphasize, as Mayo does, that this is relative to the parameters $X$, $L$, and $M$: I am not claiming that $\square H_0$ confirms $\tau$; only the empirical adequacy of $\tau$ with respect to the experiment $X$ is being confirmed.

\subsection{Two One-Sided Testing Satisfies the Desiderata for Mayo's Severe Testing.}

Mayo and other error statisticians instead advocate for a modern recasting of NHST---severe testing---as the appropriate framework guiding the use of statistical methods.  Central to the error statistician is the question:

\begin{quote}
When do data $E$ provide good evidence for/a good test of hypothesis $\tau$?
\end{quote}

The error statistician will invoke some form of a \textbf{Severity Principle} to answer this question:
\begin{quote}
(Weak Severity Principle) Data $E$ \textbf{does not} provide good evidence for $\tau$ if $E$ is the result of a \textbf{test procedure} $T$ with very low probability of uncovering the falsity of $\tau$ \cite{mayo2010introduction} (Page 21). 
\end{quote}

A converse is given by:
\begin{quote}
(Full Severity Principle) Data $E$ provides good evidence for $\tau$ to the extent that test $T$ has been \textbf{severely passed} by $\tau$ \cite{mayo2010introduction} (Page 21).
\end{quote}

The error statistician naturally asks \textit{which} hypotheses are amenable to error-theoretic analysis. This question is of utmost importance as the Full Severity Principle suggests the following account of scientific content: the hypotheses $\tau$ that have scientific content are precisely those which are severely testable.

The notion of severe testing as described by Mayo is defined as follows:
\begin{definition}
A hypothesis $\tau$ passes a severe test relative to experiment $X$ with data $E$ if (and only if):
\begin{enumerate}
\item $E$ agrees with or ``fits'' $\tau$ (for a suitable notion of fit), and
\item experiment $X$ would (with very high probability) have produced a result that fits $\tau$ less well than $E$ does, if $\tau$ were false or incorrect. \cite{mayo2005evidence} (Page 99)\qedhere
\end{enumerate}
\end{definition}

I propose the following model of severe testing. Let $\tau$ be a predictive theory in the sense described in the prior section. The predictions of $\tau$ can be assessed by a loss function $L(\cdot,\cdot)$, and margin of adequacy $M$. The \textit{empirical adequacy hypothesis} of $\tau$ relative to $L$ and $M$ is: $H_\tau: |\theta_L| \leq M$. I claim that if the experiment yields $\Vdash \square H_\tau$, we may interpret that as saying that $\tau$ has passed a severe test (to the extent that the loss function $L$ and adequacy margin $M$ are acceptably stringent). 

First, we argue that $\square H_\tau$ entails that $E$ fits with $\tau$. This is immediate: a truth condition for $\Vdash H_0$ is that the confidence interval loss statistic $\theta_L$ is wholly contained within the margin of adequacy interval $[-M,M]$.

Second, we argue that this procedure validates the second clause. Reordering the clause, it reads: 

\begin{equation}
\begin{array}{l}
\underbrace{\tau \text{ false or incorrect}}_\text{(a)} \Rightarrow  \\
\underbrace{\text{The probability that the experiment would have produced a result more discrepant with }\tau \text{ is very high.}}_\text{(b)} 
\end{array}
\end{equation}

In this context, I interpret (a) as meaning that the loss statistic $\theta$ is not in $H_0$, i.e. that $\theta \in H_1$. At first blush this may seem to be a misrendering of the intention of the clause (a). Even if so, it seems to be a necessary amendment: for any theory $\tau$, we can construct an alternate theory $\tau'$ such that $\tau$ and $\tau'$ agree about the results of a specific experiment but disagree regarding other outcomes. Unless the experiment $E$ were exhaustive of all of $\tau$'s predictions, such $\tau'$ exist. Therefore, there exist nearby theories for which the probability expressed in (b) will be low---or even zero. A strict interpretation of (a) renders severe testability impossible.

Regarding clause (b), and assuming that by (a) we mean $\theta_L \in H_1$, we interpret a result more discrepant from $\tau$ than $\square H_0$ as meaning the experiment returned $\neg\square H_0$, that is, either $\square H_1$ or $\lozenge H_0 \wedge \lozenge H_1$. This is a plausible interpretation of (b), as in both $\square H_1$ or $\lozenge H_0 \wedge \lozenge H_1$ the confidence interval for $\theta_L$ intersects $H_1$ nontrivially and is therefore more discrepant with $\tau$ than $\square H_1$. Moreover, since for $\theta \in H_1$ we have $\PP_{\theta}(\square H_0) \leq \alpha$, 
\[ \PP_{\theta} (\neg \square H_0) = \PP_{\theta}(\square H_1 \vee (\lozenge H_0 \wedge \lozenge H_1)) \geq 1-\alpha. \]
Thus, by specifying a sufficiently small $\alpha$, the probability statement second prong in the severe testing criteria can be made arbitrarily high.


\section{Conclusion}

In this paper, a modal logic of short-run frequentist confirmation was developed by leveraging the duality between hypothesis testing and statistical estimation. This formal equivalence allows one to supplant decision rules based on the p-value achieving some threshold with set-theoretic relationships between a suitably constructed confidence region and the hypothesis in question. As developed here, the approach to the hypothesis testing grounded in properties of confidence regions enjoys many felicitous properties that avoid epistemic problems common to multiple accounts of frequentist hypothesis testing.

First, there is symmetry between the null and alternative hypotheses. There is no need to select one of the hypotheses as doxastically default or as the subject of falsificational scrutiny. Rather, both null and alternative can be assessed simultaneously in light of the evidence. This process does not require any multiple testing adjustment, since $H_1$ is uniquely determined by the choice of $H_0$: $H_1 = H_0^c$. This symmetrization better accommodates what one might call exploratory inquiry, wherein the investigator harbors no strong attitude toward either $H_0$ or $H_1$.

Second, the underlying logic is trivalent, with the output of a test being characterized by which of three mutually exclusive unnested modal sentences is satisfied. The values correspond to \textit{decisive} evidence in favor of $H_0$, \textit{decisive} evidence in favor of $H_1$, or \textit{indecision} between $H_0$ and $H_1$. This decision procedure is not infallible, but does enjoy quantifiable rates of decisive error bounded above by the level of the confidence region. With respect to the formal semantics outlined in Appendix \ref{kripke-appendix}, Proposition \ref{problem-induction} formally expresses the problem of induction for this logic. 

Third, the shift to set-theoretic relationships between confidence regions and hypotheses allows us to exactly characterize the confirmable hypotheses based on their underlying topological structure. We found that standard two-sided hypotheses are unconfirmable but that equivalence and noninferiority tests are confirmable. This provides formal justification for Lakens' claim that two-sided hypotheses are not confirmable. 

This has broad implications for scientific and statistical practice. It suggests that two-sided tests can be useful for questions of refutation or falsification, but \textit{not} for confirmation. Therefore, strategies for the statistical confirmation of theories must be developed. In the final section of this paper, one such strategy was developed and was argued to be a model of Mayo's desiderata for severe testing. To circumvent the unconfirmability of two-sided hypotheses, the empirical adequacy of a theory $\tau$ with respect to a loss function $L$ and adequacy margin $M$ is subjected to a Two One-Sided Test of the hypothesis $H_\tau: |\theta_L| \leq M$. If $H_\tau$ is confirmed, then $\tau$ has passed a severe test. Crucially, this test requires a specification of both the loss function $L$ and adequacy margin $M$. The suitability of the  choice of these parameters is ultimately normative, and these choices should be justified by the investigator. 

\newpage

\appendix

\section{Kripke Semantics and Nested Modalities} \label{kripke-appendix}
The modal semantics can be construed as an example of Kripke semantics. We define it as follows:

\begin{definition}
    Let $\Theta$ be a topological space, $c$ a confidence function, and $\Omega$ a sample space. Evidence $E$ is an element $E\in \Omega^{<\omega}$.

    A possible world is specified by a pair of $\theta_w\in \Theta$ and $E_w \in \Omega^{<\omega}$: $w = (\theta_w, E_w)$.

    The ``$c$-plausibility'' accessibility relation between two worlds is given by

\begin{equation}
    w\,R_c\,v \iff \theta_v \in c(E_w)
\end{equation}

The ``extendibility'' accessibility relation is given by

\begin{equation}
    w\,R_E\,v \iff E_w\subseteq E_v
\end{equation}

Let $H \in 2^{\Theta}$. We say that $v\Vdash H$ just in case $\theta_v \in H$.
\end{definition}

We represent the corresponding modal operators of these accessibility relations by $\square_c$ and $\square_E$ respectively.

It is straightforward to verify that the Kripke semantics of $\square_c$ agrees with the semantics of unnested modal sentences given in Section \ref{semantics-naive}. 

\begin{proposition}
    Let $H\subseteq \Omega$. Then for every $w$, $w\Vdash \square_c H$ if and only if $c(E_w) \subseteq H$.
\end{proposition}
\begin{proof}
    By definition, $w\Vdash \square_c H$ if and only if, for all $v$ such that $w R_c v$, $v\Vdash H$. By definition of $R_c$, if $w R_c v$ then $\theta_v\in c(E_w)$. Therefore $w\Vdash \square_c H$ if and only if $\theta_v \in H$ for all $v$ accessible by $w$, if and only if $c(E_w) \subseteq H$.
\end{proof}

Whereas nested modalities were left undefined in Section \ref{semantics-naive}, this frame semantics yields truth conditions for nested modal formulas.

Under mild conditions, we can make sense of a key aspect of statistical epistemology:

\begin{definition}
    We say that a confidence function $c$ satisfied the \textbf{precise extension property (PEP)} provided that for each $\theta\in \Theta$ and $U\in \OO(\Theta)$ containing $\theta$ and $E\in \Omega^{<\omega}$ there exists an $E' \supseteq E$ such that $\theta\in c(E') \subseteq U$. 
\end{definition}

Wald confidence intervals satisfy the PEP. Under the assumption of the PEP, we can precisely formulate epistemic limits to statistical confirmation. In other words, we can prove a formal expression of the problem of induction specific to hypothesis testing.

\begin{proposition} \label{problem-induction}
    Let $H \subseteq \Theta$ be nonempty. Suppose further that $c$ satisfies the PEP. Then for every $w$
    \begin{equation}
        w \Vdash \lozenge_E \lozenge_c H.
    \end{equation}

    In particular, if $H$ and $H^c$ are both nonempty, then for every $w$
    \[w \Vdash (\lozenge_E \lozenge_c H) \wedge (\lozenge_E \lozenge_c H^c). \]
\end{proposition}
\begin{proof}

Suppose $w=(\theta_w, E_w)$ is a possible world. Suppose $\theta_0 \in H$. By the PEP, for every $\theta_v \in c(E_w)$ there is some extension $E' \supseteq E_w$ such that $\theta_0 \in c(E_v)$. Let $v = (\theta_0, E')$. Then $w R_c v$, realizing $\lozenge_c H$. Thus 
\[  w \Vdash \lozenge_E \lozenge_c H.\]

\end{proof}

That is, it is always possible to extend evidence to alter our position from $\square_c H$ to $\lozenge_c \neg H$.

\section{Confirmation of Point Hypotheses}

The confirmability of hypotheses through frequentist methods is a problem that has drawn considerable interest. Prior work on the characterization of confirmable or verifiable hypotheses has been done by Genin.

In the context of Two One-Sided Testing, a form of Equivalence Testing,  Lakens' makes the claim that ``it is statistically impossible to support the hypothesis that a true effect size is exactly zero.'' \cite{lakens2017equivalence}.

One possible argument in favor of this position is that the hypothesis $\theta = 0$ is \textit{sharp}: for any given piece of data sufficiently compatible with $\theta = 0$, there exist $\theta'$ near $\theta$ also compatible with the data. 

One might argue this using standard confidence intervals. For example, when comparing the difference in proportions between two subgroups $A$ and $B$, one might be tempted to say that evidence in favor of $H$ is given just in case the complement $H^c$ is rejected. Under the duality between hypothesis testing and confidence interval estimation, a hypothesis $J$ is rejected in case event $E$ is witnessed just in case $J \cap c(E) = \varnothing$. If one is to equate confirmation of $H$ with rejection of the complement $H^c$, $H$ is confirmed just in case $c(E) \subseteq H$.

The typical confidence intervals employed in practice are either open subsets of the space of model parameters (e.g. Wald confidence intervals $(l(E),u(E))$ or have nonempty interior (e.g. bootstrap confidence intervals $[l(E),u(E)]$). In this context, the irrefutability of the negation of a point hypothesis is immediate on topological grounds: a point hypothesis $H = \{\theta_0\}$ has empty interior while confidence intervals have nonempty interiors. Since having nonempty interior is preserved under supersets, point sets are not confirmable under this inferential schema.

However, this discussion relies on a crucial premise: why assume in the first place that a confidence region necessarily have open interior? More broadly speaking, we investigate how sensitive to the choice of a confidence region process $c$ the question of confirmation is.

\subsection{Rigged Confidence Regions}

We now exhibit a construction of confidence region processes that have the possibility of confirming \textit{arbitrary} hypotheses while retaining both the confidence property and the stronger topological confidence property. 

Let $\Theta$ be a set of probability measures on a sigma algebra $\Sigma$ on $\Omega$ topologized with setwise convergence. Suppose that $c: \Sigma \to 2^{\Theta}$ satisfies the confidence property at level $\alpha/2$ and that each $c(E)$ is an open subset of $\Theta$. Thus $c$ satisfies the topological confidence property at level $\alpha/2$.

Let $H \subset 2^{\Theta}$ be an arbitrary hypothesis and $E_0$ an event in $\Sigma$. We define the $(E_0, H_0)$-rigged confidence region $c^{*}$ as follows:
\[c^*(E) =
    \begin{cases}
        c(E) & \text{if } E \neq E_0\\
        H & \text{if } E = E_0.
    \end{cases}
    \]

Suppose that there exists an event $E_0$ such that 
\[ 0 < \limsup_{\theta\in \Theta} (\PP_{\theta}(E_0) ) = \gamma < \frac{\alpha}{2}.\]

Then the function $c^*$ satisfies the $\alpha$ topological confidence property, since for all $\theta$
\begin{equation}
\begin{array}{lcl}
    \PP_{\theta}(\theta \notin c(E) \text{ or } c(E) \text{ has empty interior} )& \leq & \PP_{\theta}(\theta \notin c(E)) + \PP_{\theta}(c(E) \text{ has empty interior}) \\
    
    & \leq & \frac{\alpha}{2} + \gamma \\ 
    &\leq& \alpha
    \end{array}
\end{equation}

\begin{example}
Let $\Omega = 2^{7}$ and $\Sigma = 2^{2^7}$. Let $\Theta = [0,1]$ and interpret $\PP_{\theta}(E_s) = \theta^{|s|_1} (1-\theta)^{|s|_0}$. Let $\alpha = 0.05$ and let $c$ be the Wald confidence interval at level $\alpha' = 0.04$. Consider the event $E_{1011000}$. Then
\begin{equation}
\begin{array}{lcl}
    \limsup_{\theta\in \Theta} (\PP_{\theta}(E_{1011000}) ) & = & \left(\frac{3}{7}\right)^3\left(\frac{4}{7}\right)^4 \\ 
    & \approx & 0.0084
\end{array}
\end{equation}
since $\theta = \frac{3}{7}$ maximizes the probability of the event $E_{1011000}$.

Then the rigged confidence region $c^*$ satisfies the topological confidence property at level $\alpha^* \leq 0.0485$ by Inclusion-Exclusion. \qed
\end{example}

In other words, one can rig an open-valued confidence region of level $\alpha/2$ to confirm 



\subsection{The Topology of the Confidence Property}

Let $\Theta$ be a collection of 
measures on a $\sigma$-algebra $\Sigma$ on $\Omega$, and let $c: \Sigma \to 2^\Theta$ be a \textit{confidence region} function. For example, if $\Omega = 2^{\omega}$ and $\Sigma$ is the $\sigma$-algebra generated by basic open sets in the discrete topology.

\begin{definition}
    We say that $c$ has the confidence property at level $\alpha$ provided that every for every $E$ the event $\theta \in c(E)$ is in $\Sigma$ and
    \[ \inf\limits_{\theta\in\Theta} \PP_{\theta}(\theta\in c(E)) \geq 1-\alpha. \] \qed
\end{definition}

In case $\Theta$ has topological structure, the confidence property imposes constraints on the topological structure of the confidence region.
We assume that $\Theta$ is topologized by setwise(or ``strong'') convergence of measures; that is, 
   \[\theta_n \rightarrow_{sw} \theta \Leftrightarrow \PP_{\theta_n}(A) \rightarrow \PP_{\theta} (A) \text{ for all } A\in \Sigma^{\otimes n}.\]
In other words, setwise convergence holds just in case the limit of probabilities of each event $A$ converge to the probability of the limit. 

This is a natural topology for statistical inference: a sequence of measures converges just in case the measures assigned to each specific event converge. This does not generally hold in the case of weak convergence. For example, the Dirac distribution at $\frac{1}{n}$, $\delta_{\frac{1}{n}}$ converges weakly to $\delta_0$ but does not converge setwise since $\delta_{\frac{1}{n}}(\{0\}) = 0$ for all $n$ but $\delta_0(\{0\}) = 1$.

\begin{example}
    For example that $\Theta = [0,1]$ on $\Omega = 2^{\omega}$. Let \[E_s = \{x\in \Omega\,|\, s \text{ is a prefix of } x \} \] and let $|s|_1 = \text{Number of 1's in }  s$ and $|s|_0 = \text{Number of 0's in }  s$. Then
    $\PP_{\theta}(E_s) = \theta^{|s|_1}\times (1-\theta)^{|s|_0}.$

    Then the topology $(\Theta, sw)$ is isomorphic to the Euclidean topology $(\Theta, |\cdot|)$. 
    \qed
\end{example} 

Having the confidence property constrains the structure of $c$.

\begin{theorem}\label{thm:topological_confidence}
    Suppose that $\Theta$ is topologized by setwise topology and has no isolated points.
    Let $c: \Sigma \to 2^{\Theta}$ have the confidence property at level $\alpha$. Then for all $\theta$, \[\PP_{\theta}(c(E) \text{ has nonempty interior} )\geq 1-\alpha.\]

\end{theorem}
\begin{proof}
Suppose for contradiction that there is a $\theta_0$ such that $\PP_{\theta_0}(c(E) \text{ has empty interior} )> \alpha.$ By countable additivity there exist finitely many events $E_1,\dots, E_k \in \Sigma$ such that $c(E_i)$ has empty interior and $\PP_{\theta}\left(\bigcup E_i\right) > \alpha$. Since $\Theta$ has no isolated points, there exists a neighborhood $U$ of $\theta$ containing a convergent sequence $\mu_k \to \theta$ with $\PP_{\mu_k}\left(\bigcup E_i\right) > \alpha$ for all $k$. Since each $c(E_i)$ has nonempty interior, the relative complement $U\setminus \bigcup E_i$ is dense in $U$. Therefore there exists some $\mu \in U$ such that
\begin{enumerate}
    \item $\PP_{\mu}\left(\bigcup E_i\right) > \alpha$ and
    \item $\mu \notin c(E_i)$ for any $i\leq k$.
\end{enumerate}
By construction, the event $\mu \notin c(E)$ contains $\bigcup E_i$, so 
\begin{equation}
\begin{array}{lcl}
    \PP_{\mu}(\mu \notin c(E)) & \geq  &\PP_{\mu}\left(\bigcup E_i\right) \\
    & > & \alpha,
\end{array}
\end{equation} 
violating the confidence property at level $\alpha$.
\end{proof}

As a corollary,

\begin{corollary} 
    Let $c: \Sigma \to 2^{\Theta}$ have the confidence property at level $\alpha$. Then the \textbf{topological confidence property at level $2\alpha$} holds: for all $\theta$, \[\PP_{\theta}(\theta \in c(E) \text{ and } c(E) \text{ has nonempty interior} )\geq 1-2\alpha.\]
\end{corollary}

\begin{proof}
    By Theorem \ref{thm:topological_confidence} and the Inclusion-Exclusion Principle.
\end{proof}

Thus, with high probability the confidence region contains an open subset of $\Theta$, making it unlikely to confirm hypotheses with empty interior.

\newpage

\printbibliography

@article{mayo2010introduction,
	title        = {Introduction and Background},
	author       = {Mayo, Deborah G and Spanos, Aris},
	year         = 2010,
	journal      = {Error and Inference: Recent Exchanges on Experimental Reasoning, Reliability, and the Objectivity and Rationality of Science},
	publisher    = {Cambridge University Press},
	pages        = 28
}

@book{earman1992bayes,
	title        = {Bayes or Bust? A Critical Examination of Bayesian Confirmation Theory},
	author       = {Earman, John},
	year         = 1992,
	publisher    = {Cambridge, ma: MIT Press}
}

@incollection{mayo2005evidence,
	title        = {Evidence as Passing Severe Tests: Highly Probable Versus Highly Probed Hypotheses},
	author       = {Mayo, Deborah G},
	year         = 2005,
	booktitle    = {Scientific Evidence: Philosophical Theories and Applications},
	publisher    = {JHU Press},
	pages        = {95--127},
	editor       = {Achinstein, Peter},
	chapter      = 6
}

@book{mayo2018severe,
	title        = {Statistical Inference as Severe Testing},
	author       = {Mayo, Deborah G},
	year         = 2018,
	publisher    = {Cambridge University Press}
}

@book{wasserman2006all,
	title        = {All of Statistics},
	author       = {Wasserman, Larry},
	year         = 2004,
	publisher    = {Springer Science \& Business Media}
}

@book{smithson2003confidence,
  title={Confidence intervals},
  author={Smithson, Michael},
  number={140},
  year={2003},
  publisher={Sage}
}

@book{keener2010theoretical,
  title={Theoretical statistics: Topics for a core course},
  author={Keener, Robert W},
  year={2010},
  publisher={Springer}
}

@article{lakens2017equivalence,
  title={Equivalence tests: A practical primer for t tests, correlations, and meta-analyses},
  author={Lakens, Dani{\"e}l},
  journal={Social psychological and personality science},
  volume={8},
  number={4},
  pages={355--362},
  year={2017},
  publisher={Sage Publications Sage CA: Los Angeles, CA}
}

@book{Munkres2000Topology,
  author = {Munkres, James R.},
  day = 07,
  edition = 2,
  howpublished = {Hardcover},
  isbn = {0131816292},
  month = jan,
  publisher = {Prentice Hall, Inc.},
  title = {{Topology}},
  url = {http://www.worldcat.org/isbn/0131816292},
  year = 2000
}

@book{lehmann_testing_2022,
	location = {Cham},
	title = {Testing Statistical Hypotheses},
	isbn = {9783030705770 9783030705787},
	url = {https://link.springer.com/10.1007/978-3-030-70578-7},
	series = {Springer Texts in Statistics},
	publisher = {Springer International Publishing},
	author = {Lehmann, E.L. and Romano, Joseph P.},
	urldate = {2023-06-07},
	date = {2022},
	langid = {english},
	doi = {10.1007/978-3-030-70578-7},
}

@article{frick_accepting_1995,
	title = {Accepting the null hypothesis},
	volume = {23},
	issn = {1532-5946},
	url = {https://doi.org/10.3758/BF03210562},
	doi = {10.3758/BF03210562},
	pages = {132--138},
	number = {1},
	journaltitle = {Memory \& Cognition},
	shortjournal = {Memory \& Cognition},
	author = {Frick, Robert W.},
	urldate = {2023-06-02},
	date = {1995-01-01},
	langid = {english},
}

@article{genin_topology_2017,
	title = {The Topology of Statistical Verifiability},
	volume = {251},
	issn = {2075-2180},
	url = {http://arxiv.org/abs/1707.09378},
	doi = {10.4204/EPTCS.251.17},
	pages = {236--250},
	journaltitle = {Electronic Proceedings in Theoretical Computer Science},
	shortjournal = {Electron. Proc. Theor. Comput. Sci.},
	author = {Genin, Konstantin and Kelly, Kevin T.},
	urldate = {2023-06-25},
	date = {2017-07-25},
	eprinttype = {arxiv},
	eprint = {1707.09378 [cs, math]},
}

@article{mayo_severe_2006,
	title = {Severe Testing as a Basic Concept in a Neyman-Pearson Philosophy of Induction},
	volume = {57},
	issn = {0007-0882},
	url = {https://www.jstor.org/stable/3873470},
	pages = {323--357},
	number = {2},
	journaltitle = {The British Journal for the Philosophy of Science},
	author = {Mayo, Deborah G. and Spanos, Aris},
	urldate = {2023-06-26},
	date = {2006},
}

\end{document}